\documentclass[11pt]{article}

\usepackage{amsmath,amsfonts,amssymb} 

\usepackage{amsthm} 

\usepackage{graphicx,caption,subcaption} 

\usepackage{fullpage} 
\usepackage{hyperref} 
\usepackage[parfill]{parskip} 

\allowdisplaybreaks 

\begingroup
    \makeatletter
    \@for\theoremstyle:=definition,remark,plain\do{%
        \expandafter\g@addto@macro\csname th@\theoremstyle\endcsname{%
            \addtolength\thm@preskip\parskip
            }%
        }
\endgroup

\newcommand{\cH}{{\mathcal H}}

\newcommand{\floor}[1]{{\left\lfloor{#1}\right\rfloor}}

\newtheorem{theorem}{Theorem}[section]
\newtheorem{prop}{Proposition}[section]
\newtheorem{lemma}{Lemma}[section]

\title{Distance-Uniform Graphs with Large Diameter}
\author{Mikhail Lavrov\thanks{University of Illinois at Urbana-Champaign, Department of Mathematics. E-mail: \texttt{mlavrov{@}illinois.edu}.} \and Po-Shen Loh\thanks{Carnegie Mellon University, Department of Mathematical Sciences. E-mail: \texttt{ploh{@}cmu.edu}.} \and Arnau Messegu\'e\thanks{Polytechnic University of Catalonia, Computer Science Department. E-mail: \texttt{messegue{@}cs.upc.edu}.}}

\begin{document}

\maketitle

\begin{abstract}
	An $\epsilon$-distance-uniform graph is one in which from every vertex, all but an $\epsilon$-fraction of the remaining vertices are at some fixed distance $d$, called the critical distance. We consider the maximum possible value of $d$ in an $\epsilon$-distance-uniform graph with $n$ vertices. We show that for $\frac1n \le \epsilon \le \frac1{\log n}$, there exist $\epsilon$-distance-uniform graphs with critical distance $2^{\Omega(\frac{\log n}{\log \epsilon^{-1}})}$, disproving a conjecture of Alon et al.\ that $d$ can be at most logarithmic in $n$. We also show that our construction is best possible, in the sense that an upper bound on $d$ of the form $2^{O(\frac{\log n}{\log \epsilon^{-1}})}$ holds for all $\epsilon$ and $n$.
\end{abstract}

\section{Introduction}

We say that an $n$-vertex graph is \emph{$\epsilon$-distance-uniform} for some parameter $\epsilon>0$ if there is a value $d$, called the \emph{critical distance}, such that, for every vertex $v$, all but at most $\epsilon n$ of the other vertices are at distance exactly $d$ from $v$. Distance-uniform graphs exist for some, but not all, possible triplets $(n, \epsilon, d)$; a trivial example is the complete graph $K_n$, which is distance-uniform with $\epsilon = \frac1n$ and $d=1$. So it is natural to try to characterize which triplets $(n,\epsilon, d)$ are realizable as distance-uniform graphs.

The notion of distance uniformity is introduced by Alon, Demaine, Hajiaghayi, and Leighton in \cite{alon13}, motivated by the analysis of network creation games. It turns out that equilibria in a certain network creation game can be used to construct distance-uniform graphs. As a result, understanding distance-uniform graphs tells us which equilibria are possible.

\subsection{From network creation games to distance uniformity}

The use of the Internet has been growing significantly in the last few decades. This fact has motivated theoretical studies that try to capture properties of Internet-like networks into models. Fabrikant et al. \cite{fabrikant03} proposed one of these first models, the so called \emph{sum classic network creation game} (or abbreviated sum classic) from which variations (like \cite{bilo15}, \cite{ehsani15}) and extensions of it (like  \cite{bilo15max}, \cite{brandes08}) have been considered in the subsequent years. 

Although all these models try to capture different aspects of Internet, all of them can be identified as \emph{strategic games}: every agent or node (every player in the game) buys some links (every player picks an strategy) in order to be connected in the network formed by all the players (the strategic configurations formed as a combination of the strategies of every player) and tries to minimize a cost function modeling their needs and interests. 

All these models together with their results constitute a whole subject inside game theory and computer science that stands on its own: the field of \emph{network creation games}. Some of the most relevant concepts discussed in network creation games are \emph{optimal network}, \emph{Nash equilibria} and the \emph{price of anarchy}, among others. 

An optimal network is the outcome of a configuration having minimum overall cost, that is, the sum of the costs of every player has the minimum possible value. A Nash equilibrium is a configuration where each player cannot strictly decrease his cost function given that the strategies of the other players are fixed.  The price of anarchy quantifies the loss in terms of efficiency between the worst Nash equilibrium (anyone having maximum overall cost) and any optimal network (anyone having minimal overall cost). 

The sum classic is specified with a set of players $N = \left\{ 1,...,n\right\}$ and a parameter $\alpha > 0$ representing the cost of establishing a link. Every player $i\in N$ wishes to be connected in the resulting network, then the strategy $s_i \in \mathcal{P}(N \setminus \left\{i \right\})$ represents the subset of players to which $i$ establishes links. Then considering the tuple of the strategies for every player $s=(s_1,...,s_n)$ (called a \emph{strategy profile}) the \emph{communication network} associated to $s$, noted as $G[s]$, is defined as the undirected graph having $N$ as the set of vertices and the edges $(i,j)$ iff $i \in s_j $ or $j \in s_i$. The communication network represents the resulting network obtained after considering the links bought for every node. Then the cost function for a strategy profile $s = (s_1,...,s_n)$ has two components: the \emph{link cost} and the \emph{usage cost}. The link cost for a player $i \in N$ is $\alpha |s_i|$ and it quantifies the cost of buying $|s_i|$ links. In contrast, the usage cost for a player $i$ is $\sum_{j \neq i} d_{G[s]}(i,j)$. Therefore, the total cost incurred for player $i$ is $c_i(s)=\alpha |s_i|+\sum_{j \neq i} d_{G[s]}(i,j)$. 

On the other hand, a given undirected graph $G$ in the \emph{sum basic network creation game} (or abbreviated sum basic) is said to be in equilibrium iff, for every edge $(i,j) \in E(G)$ and every other player $k$, the player $i$ does not strictly decrease the sum of distances to the other players by swapping the edge $(i,j)$ for the edge $(i,k)$.

At first glance, the sum basic could be seen as the model obtained from the sum classic when considering only deviations that consists in swapping individual edges. However,  in any Nash equilibrium for the sum classic, only one of the endpoints of any edge has bought that specific edge so that just one of the endpoints of the edge can perform a swap of that specific edge. Therefore, one must be careful when trying to translate a property or result from the sum basic to the sum classic. 

In the sum classic game it has been conjectured that the price of anarchy is constant (asymptotically) for any value of $\alpha$. Until now this conjecture has been proved true for $\alpha = O(n^{1-\epsilon})$ with $\epsilon \geq 1/\log n $ (\cite{demaine12}) and for $\alpha >9n$ (\cite{alvarez17}). In \cite{demaine12} it is proved that the price of anarchy is upper bounded by the diameter of any Nash equilibrium. This is why the diameter of equilibria in the sum basic is studied. 

In \cite{alon13}, the authors show that sufficiently large graph powers of an equilibrium graph in the sum basic model will result in distance-uniform graphs; if the critical distance is large, then the original equilibrium graph in the sum basic model imposed a high total cost on its nodes. In particular, it follows that if $\epsilon$-distance-uniform graphs had diameter $O(\log n)$, the diameter of equilibria for the sum basic would be at most $O(\log^3 n)$.

\subsection{Previous results on distance uniformity}

This application motivates the already natural question: in an $\epsilon$-distance-uniform graph with $n$ vertices and critical distance $d$, what is the relationship between the parameters $\epsilon$, $n$, and $d$? Specifically, can we derive an upper bound on $d$ in terms of $\epsilon$ and $n$? Up to a constant factor, this is equivalent to finding an upper bound on the diameter of the graph, which must be between $d$ and $2d$ as long as $\epsilon < \frac12$.

Random graphs provide one example of distance-uniform graphs. In \cite{bollobas81}, Bollob\'as shows that for sufficiently large $p = p(n)$, the diameter of the random Erd\H{o}s--R\'enyi random graph $\mathcal G_{n,p}$ is asymptotically almost surely concentrated on one of two values. In fact, from every vertex $v$ in $\mathcal G_{n,p}$, the breadth-first search tree expands by a factor of $O(np)$ at every layer, reaching all or almost all vertices after about $\log_r n$ steps. Such a graph is also expected to be distance-uniform: the biggest layer of the breadth-first search tree will be much bigger than all previous layers.

More precisely, suppose that we choose $p(n)$ so that the average degree $r = (n-1)p$ satisfies two criteria: that $r \gg (\log n)^3$, and that for some $d$, $r^d/n - 2 \log n$ approaches a constant $C$ as $n \to \infty$. Then it follows from Lemma~3 in \cite{bollobas81} that (with probability $1-o(1)$) for every vertex $v$ in $\mathcal G_{n,p}$, the number of vertices at each distance $k < d$ from $v$ is $O(r^k)$. It follows from Theorem~6 in \cite{bollobas81} that the number of vertex pairs in $\mathcal G_{n,p}$ at distance $d+1$ from each other is Poisson with mean $\frac12 e^{-C}$, so there are only $O(1)$ such pairs with probability $1-o(1)$. As a result, such a random graph is $\epsilon$-distance-uniform with $\epsilon = O(\frac{\log n}{r})$, and critical distance $d = \log_r n + O(1)$.

This example provides a compelling image of what distance-uniform graphs look like: if the breadth-first search tree from each vertex grows at the same constant rate, then most other vertices will be reached in the same step. In any graph that is distance-uniform for a similar reason, the critical distance $d$ will be at most logarithmic in $n$. In fact, Alon et al.\ conjecture that all distance-uniform graphs have diameter $O(\log n)$. 

Alon et al.\ prove an upper bound of $O(\frac{\log n}{\log \epsilon^{-1}})$ in a special case: for $\epsilon$-distance-uniform graphs with $\epsilon<\frac14$ that are Cayley graphs of Abelian groups. In this case, if $G$ is the Cayley graph of an Abelian group $A$ with respect to a generating set $S$, one form of Pl\"unnecke's inequality (see, e.g., \cite{tao06}) says that the sequence 
\[
	|\underbrace{S + S + \dots + S}_k|^{1/k}
\] 
is decreasing in $k$. Since $S, S+S, S+S+S, \dots$ are precisely the sets of vertices which can be reached by $1, 2, 3, \dots$ steps from 0, this inequality quantifies the idea of constant-rate growth in the breadth-first search tree; Theorem~15 in~\cite{alon13} makes this argument formal.

\subsection{Our results}

In this paper, we disprove Alon et al.'s conjecture by constructing distance-uniform graphs that do not share this behavior, and whose diameter is exponentially larger than these examples. We also prove an upper bound on the critical distance (and diameter) showing our construction to be best possible in one asymptotic sense. Specifically, we show the following two results:

\begin{theorem}
\label{thm:intro-upper}
In any $\epsilon$-distance-uniform graph with $n$ vertices, the critical distance $d$ satisfies
\[
	d = 2^{O\left(\frac{\log n}{\log \epsilon^{-1}}\right)}.
\]
\end{theorem}

\begin{theorem}
\label{thm:intro-lower}
For any $\epsilon$ and $n$ with $\frac1n \le \epsilon \le \frac1{\log n}$, there exists an $\epsilon$-distance-uniform graph on $n$ vertices with critical distance
\[
	d = 2^{\Omega\left(\frac{\log n}{\log \epsilon^{-1}}\right)}.
\]
\end{theorem}

Note that, since a $\frac1{\log n}$-distance-uniform graph is also $\frac12$-distance-uniform, Theorem~\ref{thm:intro-lower} also provides a lower bound of $d = 2^{\Omega(\frac{\log n}{\log \log n})}$ for any $\epsilon > \frac1{\log n}$.

Combined, these results prove that the maximum critical distance is $2^{\Theta(\frac{\log n}{\log \epsilon^{-1}})}$ whenever they both apply. A small gap remains for sufficiently large $\epsilon$: for example when $\epsilon$ is constant as $n \to \infty$. In this case, Theorem~\ref{thm:intro-upper} gives an upper bound on $d$ which is polynomial in $n$, while the lower bound of Theorem~\ref{thm:intro-lower} grows slower than any polynomial.

The family of graphs used to prove Theorem~\ref{thm:intro-lower} is interesting in its own right. We give two different interpretations of the underlying structure of these graphs. First, we describe a combinatorial game, generalizing the well-known Tower of Hanoi puzzle, whose transition graph is $\epsilon$-distance-uniform and has large diameter. Second, we give a geometric interpretation, under which each graph in the family is the skeleton of the convex hull of an arrangement of points  on a high-dimensional sphere.

\section{Upper bound}

\newcommand{\Nexact}[2]{\Gamma_{#1}(#2)}
\newcommand{\Natmost}[2]{N_{#1}(#2)}

For a vertex $v$ of a graph $G$, let $\Nexact{r}{v}$ denote the set $\{w \in V(G) \mid d(v,w) = r\}$: the vertices at distance exactly $r$ from $v$. In particular, $\Nexact{0}{v} = \{v\}$ and $\Nexact{1}{v}$ is the set of all vertices adjacent to $v$. Let
\[
	\Natmost{r}{v} = \bigcup_{i=0}^r \Nexact{i}{v}
\]
denote the set of vertices within distance at most $r$ from $v$.

Before proceeding to the proof of Theorem~\ref{thm:intro-upper}, we begin with a simple argument that is effective for an $\epsilon$ which is very small:

\begin{lemma}
\label{lemma:min-degree}
The minimum degree $\delta(G)$ of an $\epsilon$-distance-uniform graph $G$ satisfies $\delta(G) \ge \epsilon^{-1} - 1$.
\end{lemma}
\begin{proof}
Suppose that $G$ is $\epsilon$-distance-uniform, $n$ is the number of vertices of $G$, and $d$ is the critical distance: for any vertex $v$, at least $(1-\epsilon)n$ vertices of $G$ are at distance exactly $d$ from $v$.

Let $v$ be an arbitrary vertex of $G$, and fix an arbitrary breadth-first search tree $T$, rooted at $v$. We define the \emph{score} of a vertex $w$ (relative to $T$) to be the number of vertices at distance $d$ from $v$ which are descendants of $w$ in the tree $T$.

There are at least $(1-\epsilon)n$ vertices at distance $d$ from $v$, and all of them are descendants of some vertex in the neighborhood $\Nexact{1}{v}$. Therefore the total score of all vertices in $\Nexact{1}{v}$ is at least $(1-\epsilon)n$.

On the other hand, if $w \in \Nexact{1}{v}$, each vertex counted by the score of $w$ is at distance $d-1$ from $w$. Since at least $(1-\epsilon)n$ vertices are at distance $d$ from $w$, at most $\epsilon n$ vertices are at distance $d-1$, and therefore the score of $w$ is at most $\epsilon n$.

In order for $|\Nexact{1}{v}|$ scores of at most $\epsilon n$ to sum to at least $(1-\epsilon)n$, $|\Nexact{1}{v}|$ must be at least $\frac{(1-\epsilon)n}{\epsilon n} = \epsilon^{-1} - 1$.
\end{proof}

This lemma is enough to show that in a $\frac1{\sqrt n}$-distance-uniform graph, the critical distance is at most $2$. Choose a vertex $v$: all but $\sqrt n$ of the vertices of $G$ are at the critical distance $d$ from $v$, and $\sqrt n - 1$ of the vertices are at distance $1$ from $v$ by Lemma~\ref{lemma:min-degree}. The remaining uncounted vertex is $v$ itself. It is impossible to have $d \ge 3$, as that would leave no vertices at distance $2$ from $v$.

For larger $\epsilon$, the bound of Lemma~\ref{lemma:min-degree} becomes ineffective, but we can improve it by a more general argument of which Lemma~\ref{lemma:min-degree} is just a special case.

\begin{lemma}
\label{lemma:arnau}
	Let $G$ be an $\epsilon$-distance-uniform graph with critical distance $d$. Suppose that for some $r$ with $2r+1 \le d$, we have $|\Natmost{r}{v}| \ge N$ for each $v \in V(G)$. Then we have $|\Natmost{3r+1}{v}| \ge N\epsilon^{-1}$ for each $v \in V(G)$.
\end{lemma}
\begin{proof}	
Let $v$ be any vertex of $G$, and let $\{w_1, w_2, \dots, w_t\}$ be a maximal collection of vertices in $\Nexact{2r+1}{v}$ such that $d(w_i, w_j) \ge 2r+1$ for each $i \ne j$ with $1 \le i,j \le t$.

We claim that for each vertex $u \in \Nexact{d}{v}$---for each vertex $u$ at the critical distance from $v$---there is some $i$ with $1 \le i \le t$ such that $u \in \Natmost{d-1}{w_i}$. To see this, consider any shortest path from $v$ to $u$, and let $u_\pi \in \Nexact{2r+1}{v}$ be the $(2r+1)$\textsuperscript{th} vertex along this path. (Here we use the assumption that $2r+1 \le d$.) From the maximality of $\{w_1, w_2, \dots, w_t\}$, it follows that $d(w_i, u_\pi) \le 2r$ for some $i$ with $1 \le i \le t$. But then,
\[
	d(w_i, u) \le d(w_i, u_\pi) + d(u_\pi, u) \le 2r + (d - 2r-1) = d-1.
\]
So $u \in \Natmost{d-1}{w_i}$.

To state this claim differently, the sets $\Natmost{d-1}{w_1}, \dots, \Natmost{d-1}{w_t}$ together cover $\Nexact{d}{v}$. These sets are all small while the set they cover is large, so there must be many of them:
\[
	(1-\epsilon)n \le |\Nexact{d}{v}| \le \sum_{i=1}^t |\Natmost{d-1}{w_i}| \le \sum_{i=1}^t \epsilon n = t \epsilon n,
\]
which implies that $t \ge \frac{(1-\epsilon)n}{\epsilon n} = \epsilon^{-1} - 1$.

The vertices $v, w_1, w_2, \dots, w_t$ are each at distance at least $2r+1$ from each other, so the sets $\Natmost{r}{v}, \Natmost{r}{w_1}, \dots, \Natmost{r}{w_t}$ are disjoint. 

By the hypothesis of this lemma, each of these sets has size at least $N$, and we have shown that there are at least $\epsilon^{-1}$ sets. So their union has size at least $N\epsilon^{-1}$. Their union is contained in $\Natmost{3r+1}{v}$, so we have $|\Natmost{3r+1}{v}| \ge N\epsilon^{-1}$, as desired.
\end{proof}

We are now ready to prove Theorem~\ref{thm:intro-upper}. The strategy is to realize that the lower bounds on $|\Natmost{r}{v}|$, which we get from Lemma~\ref{lemma:arnau}, are also lower bounds on $n$, the number of vertices in the graph. By applying Lemma~\ref{lemma:arnau} iteratively for as long as we can, we can get a lower bound on $n$ in terms of $\epsilon$ and $d$, which translates into an upper bound on $d$ in terms of $\epsilon$ and $n$.

More precisely, set $r_1 = 1$ and $r_k = 3r_{k-1} + 1$, a recurrence which has closed-form solution $r_k = \frac{3^k - 1}{2}$. Lemma~\ref{lemma:min-degree} tells us that in an $\epsilon$-distance-uniform graph $G$ with critical distance $d$, $\Natmost{r_1}{v} \ge \epsilon^{-1}$. Lemma~\ref{lemma:arnau} is the inductive step: if, for all $v$, $\Natmost{r_k}{v} \ge \epsilon^{-k}$, then $\Natmost{r_{k+1}}{v} \ge \epsilon^{-(k+1)}$, as long as $2r_k + 1 \le d$.

The largest $k$ for which $2r_k + 1 \le d$ is $k = \floor{\log_3 d}$. So we can inductively prove that
\[
	n \ge \Natmost{r_{k+1}}{v} \ge \epsilon^{-(\floor{\log_3 d} + 1)}
\]
which can be rearranged to get
\[
	\frac{\log n}{\log \epsilon^{-1}} -1 \ge \floor{\log_3 d}.
\]
This implies that
\[
	d \le 3^{\frac{\log n}{\log \epsilon^{-1}}} = 2^{O\left(\frac{\log n}{\log \epsilon^{-1}}\right)},
\]
proving Theorem~\ref{thm:intro-upper}.

\section{Lower bound}

To show that this bound on $d$ is tight, we need to construct an $\epsilon$-distance-uniform graph with a large critical distance $d$. We do this by defining a puzzle game whose state graph has this property.

\subsection{The Hanoi game}

We define a \emph{Hanoi state} to be a finite sequence of nonnegative integers $\vec x = (x_1, x_2, \dots, x_k)$ such that, for all $i > 1$, $x_i \ne x_{i-1}$. Let 
\[
	\cH_{r,k} = \big\{ \vec x \in \{0,1,\dots, r\}^k : \vec x \mbox{ is a Hanoi state}\big\}.
\]
For convenience, we also define a \emph{proper Hanoi state} to be a Hanoi state $\vec x$ with $x_1 \ne 0$, and $\cH_{r,k}^* \subset \cH_{r,k}$ to be the set of all proper Hanoi states. While everything we prove will be equally true for Hanoi states and proper Hanoi states, it is more convenient to work with $\cH_{r,k}^*$, because $|\cH_{r,k}^*| = r^k$.

In the \emph{Hanoi game on $\cH_{r,k}$}, an initial state $\vec a \in \cH_{r,k}$ and a final state $\vec b \in \cH_{r,k}$ are chosen. The state $\vec a$ must be transformed into $\vec b$ via a sequence of moves of two types:
\begin{enumerate}
\item An \emph{adjustment} of $\vec x \in \cH_{r,k}$ changes $x_k$ to any value in $\{0,1,\dots, r\}$ other than $x_{k-1}$. For example, $(1,2,3,4)$ can be changed to $(1,2,3,0)$ or $(1,2,3,5)$, but not $(1,2,3,3)$.

\item An \emph{involution} of $\vec x \in \cH_{r,k}$ finds the longest tail segment of $\vec x$ on which the values $x_k$ and $x_{k-1}$ alternate, and swaps $x_k$ with $x_{k-1}$ in that segment. For example, $(1,2,3,4)$ can be changed to $(1,2,4,3)$, or $(1,2,1,2)$ to $(2,1,2,1)$.
\end{enumerate}

We define the Hanoi game on $\cH_{r,k}^*$ in the same way, but with the added requirement that all states involved should be proper Hanoi states. This means that involutions (or, in the case of $k=1$, adjustments) that would change $x_1$ to $0$ are forbidden.

The name ``Hanoi game'' is justified because its structure is similar to the structure of the classical Tower of Hanoi puzzle. In fact, though we have no need to prove this, the Hanoi game on $\cH_{3,k}^*$ is isomorphic to a Tower of Hanoi puzzle with $k$ disks. 

It is well-known that the $k$-disk Tower of Hanoi puzzle can be solved in $2^k-1$ moves, moving a stack of $k$ disks from one peg to another. In \cite{hinz92}, a stronger statement is shown: only $2^k-1$ moves are required to go from any initial state to any final state. A similar result holds for the Hanoi game on $\cH_{r,k}$:

\begin{lemma}
\label{lemma:hanoi-diameter}
The Hanoi game on $\cH_{r,k}$ (or $\cH_{r,k}^*$) can be solved in at most $2^k-1$ moves for any initial state $\vec a$ and final state $\vec b$.
\end{lemma}
\begin{proof}
We induct on $k$ to show the following stronger statement: for any initial state $\vec a$ and final state $\vec b$, a solution of length at most $2^k-1$ exists for which any intermediate state $\vec x$ has $x_1 = a_1$ or $x_1 = b_1$. This auxiliary condition also means that if $\vec a, \vec b \in \cH_{r,k}^*$, all intermediate states will also stay in $\cH_{r,k}^*$.

When $k=1$, a single adjustment suffices to change $\vec a$ to $\vec b$, which satisfies the auxiliary condition. 

For $k>1$, there are two possibilities when changing $\vec a $ to $\vec b$:
\begin{itemize}
\item If $a_1 = b_1$, then consider the Hanoi game on $\cH_{r,k-1}$ with initial state $(a_2, a_3, \dots, a_k)$ and final state $(b_2, b_3, \dots, b_k)$. By the inductive hypothesis, a solution using at most $2^{k-1} - 1$ moves exists. 

Apply the same sequence of adjustments and involutions in $\cH_{r,k}$ to the initial state $\vec a$. This has the effect of changing the last $k-1$ entries of $\vec a$ to $(b_2, b_3, \dots, b_k)$. To check that we have obtained $\vec b$, we need to verify that the first entry is left unchanged.

The auxiliary condition of the inductive hypothesis tells us that all intermediate states have $x_2 = a_2$ or $x_2 = b_2$. Any move that leaves $x_2$ unchanged also leaves $x_1$ unchanged. A move that changes $x_2$ must be an involution swapping the values $a_2$ and $b_2$; however, $x_1 = a_1 \ne a_2$, and $x_1 = b_1 \ne b_2$, so such an involution also leaves $x_1$ unchanged.

Finally, the new auxiliary condition is satisfied, since we have $x_1 = a_1 = b_1$ for all intermediate states.

\item If $a_1 \ne b_1$, begin by taking $2^{k-1}-1$ moves to change $\vec a$ to $(a_1, b_1, a_1, b_1, \dots)$ while satisfying the auxiliary condition, as in the first case.

An involution takes this state to $(b_1, a_1, b_1, a_1, \dots)$; this continues to satisfy the auxiliary condition.

Finally, $2^{k-1}-1$ more moves change this state to $\vec b$, as in the first case, for a total of $2^k-1$ moves.\qedhere
\end{itemize}
\end{proof}

If we obtain the same results as in the standard Tower of Hanoi puzzle, why use the more complicated game in the first place? The reason is that in the classical problem, we cannot guarantee that any starting state would have a final state $2^k-1$ moves away. With the rules we define, as long as the parameters are chosen judiciously, each state $\vec a \in \cH_{r,k}$ is part of many pairs $(\vec a, \vec b)$ for which the Hanoi game requires $2^k-1$ moves to solve. 

The following lemma almost certainly does not characterize such pairs, but provides a simple sufficient condition that is strong enough for our purposes.

\begin{lemma}
\label{lemma:hanoi-game}
The Hanoi game on $\cH_{r,k}$ (or $\cH_{r,k}^*$) requires exactly $2^k-1$ moves to solve if $\vec a$~and~$\vec b$ are chosen with disjoint support: that is, $a_i \ne b_j$ for all $i$ and $j$.
\end{lemma}
\begin{proof}
Since Lemma~\ref{lemma:hanoi-diameter} proved an upper bound of $2^k-1$ for all pairs $(\vec a, \vec b)$, we only need to prove a lower bound in this case.

Once again, we induct on $k$. When $k=1$, a single move is necessary to change $\vec a$ to $\vec b$ if $\vec a \ne \vec b$, verifying the base case.

Consider a pair $\vec a, \vec b \in \cH_{r,k}$ with disjoint support, for $k > 1$. Moreover, assume that $\vec a$ and $\vec b$ are chosen so that, of all pairs with disjoint support, $\vec a$ and $\vec b$ require the least number of moves to solve the Hanoi game. (Since we are proving a lower bound on the number of moves necessary, this assumption is made without loss of generality.)

In a shortest path from $\vec a$ to $\vec b$, every other move is an adjustment: if there were two consecutive adjustments, the first adjustment could be skipped, and if there were two consecutive involutions, they would cancel out and both could be omitted. Moreover, the first move is an adjustment: if we began with an involution, then the involution of $\vec a$ would be a state closer to $\vec b$ yet still with disjoint support to $\vec b$, contrary to our initial assumption. By the same argument, the last move must be an adjustment.

Given a state $\vec x \in \cH_{r,k}$, let its \emph{abbreviation} be $\vec x' = (x_1, x_2, \dots, x_{k-1}) \in \cH_{r,k-1}$. An adjustment of $\vec x$ has no effect on $\vec x'$, since only $x_k$ is changed. If $x_k \ne x_{k-2}$, then an involution of $\vec x$ is an adjustment of $\vec x'$, changing its last entry $x_{k-1}$ to $x_k$. Finally, if $x_k = x_{k-2}$, then an involution of $\vec x$ is also an involution of $\vec x'$. 

Therefore, if we take a shortest path from $\vec a$ to $\vec b$, omit all adjustments, and then abbreviate all states, we obtain a solution to the Hanoi game on $\cH_{r,k-1}$ that takes $\vec a'$ to $\vec b'$. By the inductive hypothesis, this solution contains at least $2^{k-1} - 1$ moves, since $\vec a'$ and $\vec b'$ have disjoint support. Therefore the shortest path from $\vec a$ to $\vec b$ contains at least $2^{k-1}-1$ involutions. Since the first, last, and every other move is an adjustment, there must be $2^{k-1}$ adjustments as well, for a total of $2^k-1$ moves.
\end{proof}

Now let the \emph{Hanoi graph $G_{r,k}^*$} be the graph with vertex set $\cH_{r,k}^*$ and edges joining each state to all the states that can be obtained from it by a single move. Since an adjustment can be reversed by another adjustment, and an involution is its own inverse, $G_{r,k}^*$ is an undirected graph.

For any state $\vec a \in \cH_{r,k}^*$, there are at least $(r-k)^k$ other states with disjoint support to $\vec a$, out of $|\cH_{r,k}^*| = r^k$ other states, forming a $\left(1 - \frac{k}{r}\right)^k > 1 - \frac{k^2}{r}$ fraction of all the states. By Lemma~\ref{lemma:hanoi-game}, each such state $\vec b$ is at distance $2^k-1$ from $\vec a$ in the graph $G_{r,k}^*$, so $G_{r,k}^*$ is $\epsilon$-distance uniform with $\epsilon = \frac{k^2}{r}$, $n = r^k$ vertices, and critical distance $d = 2^k-1$.

Having established the graph-theoretic properties of $G_{r,k}^*$, we now prove Theorem~\ref{thm:intro-lower} by analyzing the asymptotic relationship between these parameters.

\begin{proof}[Proof of Theorem~\ref{thm:intro-lower}]
Begin by assuming that $n = 2^{2^m}$ for some $m$. Choose $a$ and $b$ such that $a+b=m$ and
\[
	\frac{2^{2b}}{2^{2^a}} \le \epsilon < \frac{2^{2(b+1)}}{2^{2^{a-1}}},
\]
which is certainly possible since $\frac{2^0}{2^{2^m}} = \frac1n \le \epsilon$ and $\frac{2^{2m}}{2^{2^0}} > 1 \ge \epsilon$. Setting $r = 2^{2^a}$ and $k = 2^b$, the Hanoi graph $G_{r,k}^*$ has $n$ vertices and is $\epsilon$-distance uniform, since $\frac{k^2}{r} \le \epsilon$. Moreover, our choice of $a$~and~$b$ guarantees that $\epsilon < \frac{4k^2}{\sqrt{r}}$, or $\log \epsilon^{-1} \ge \frac12 \log r - 2 \log 2k$. Since $n = r^k$, $\log n = k \log r$, so
\[
	\log \epsilon^{-1} \ge \frac{1}{2k} \log n - 2 \log 2k.
\]
We show that $k \ge \frac{\log n}{6 \log \epsilon^{-1}}$. Since $\epsilon \le \frac1{\log n}$, this is automatically true if $k \ge \frac{\log n}{6 \log \log n}$, so assume that $k < \frac{\log n}{6 \log \log n}$. Then
\[
	\frac{1}{3k} \log n > 2 \log \log n > 2 \log 2k, 
\]
so
\[
	\log \epsilon^{-1} \ge \frac1{2k} \log n - 2 \log 2k > \frac{1}{2k} \log n - \frac1{3k} \log n = \frac1{6k} \log n,
\]
which gives us the desired inequality $k \ge \frac{\log n}{6 \log \epsilon^{-1}}$. The Hanoi graph $G_{r,k}^*$ has critical distance $d = 2^k - 1 = 2^{\Omega(\frac{\log n}{\log \epsilon^{-1}})}$, so the proof is finished in the case that $n$ has the form $2^{2^m}$ for some $n$.

For a general $n$, we can choose $m$ such that $2^{2^m} \le n < 2^{2^{m+1}} = \left(2^{2^m}\right)^2$, which means in particular that $2^{2^m} \ge \sqrt n$. If $\epsilon < \frac{2}{\sqrt n}$, then the requirement of a critical distance of $2^{\Omega(\frac{\log n}{\log \epsilon^{-1}})}$ is only a constant lower bound, and we may take the graph $K_n$. Otherwise, by the preceding argument, there is a $\frac{\epsilon}{2}$-distance-uniform Hanoi graph with $2^{2^m}$ vertices; its critical distance $d$ satisfies
\[
	d \ge 2^{\Omega\left(\frac{\log \sqrt{n}}{\log (\epsilon/2)^{-1}}\right)} = 2^{\Omega\left(\frac{\log n}{\log \epsilon^{-1}}\right)}.
\]
To extend this to an $n$-vertex graph, take the blow-up of the $2^{2^m}$-vertex Hanoi graph, replacing every vertex by either $\lfloor n/2^{2^m} \rfloor$ or $\lceil n/2^{2^m} \rceil$ copies. 

Whenever $v$ and $w$ were at distance $d$ in the original graph, the copies of $v$ and $w$ will be at distance $d$ in the blow-up. The difference between floor and ceiling may slightly ruin distance uniformity, but the graph started out $\frac{\epsilon}{2}$-distance-uniform, and $\lceil n/2^{2^m} \rceil$ differs from $\lfloor n/2^{2^m} \rfloor$ at most by a factor of 2. Even in the worst case, where for some vertex $v$ the $\frac{\epsilon}{2}$-fraction of vertices not at distance $d$ from $v$ all receive the larger number of copies, the resulting $n$-vertex graph will be $\epsilon$-distance-uniform.
\end{proof}

\subsection{Points on a sphere}

In this section, we identify $G_{r,k}$, the graph of the Hanoi game on $\cH_{r,k}$, with a graph that arises from a geometric construction.

Fix a dimension $r$. We begin by placing $r+1$ points on the $r$-dimensional unit sphere arbitrarily in general position (though, for the sake of symmetry, we may place them at the vertices of an equilateral $r$-simplex). We identify these points with a graph by taking the 1-skeleton of their convex hull. In this starting configuration, we simply get $K_{r+1}$.

Next, we define a truncation operation on a set of points on the $r$-sphere. Let $\delta>0$ be sufficiently small that a sphere of radius $1-\delta$, concentric with the unit sphere, intersects each edge of the 1-skeleton in two points. The set of these intersection points is the new arrangement of points obtained by the truncation; they all lie on the smaller sphere, and for convenience, we may scale them so that they are once again on the unit sphere. An example of this is shown in Figure~\ref{fig:truncation}.

\begin{figure}[h!]
        \centering
        \begin{subfigure}[b]{0.3\textwidth}
                \includegraphics[width=\textwidth]{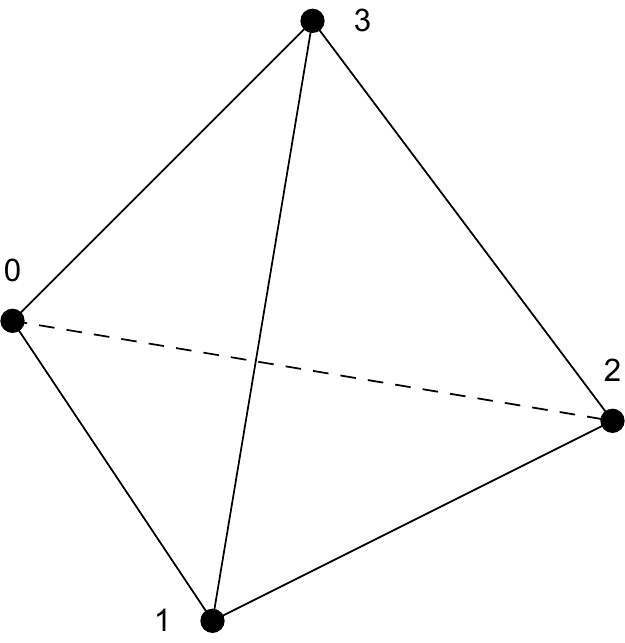}
                \caption{A tetrahedron}
                \label{fig:truncation-1}
        \end{subfigure}%
        \qquad \qquad
       \begin{subfigure}[b]{0.3\textwidth}
                \includegraphics[width=\textwidth]{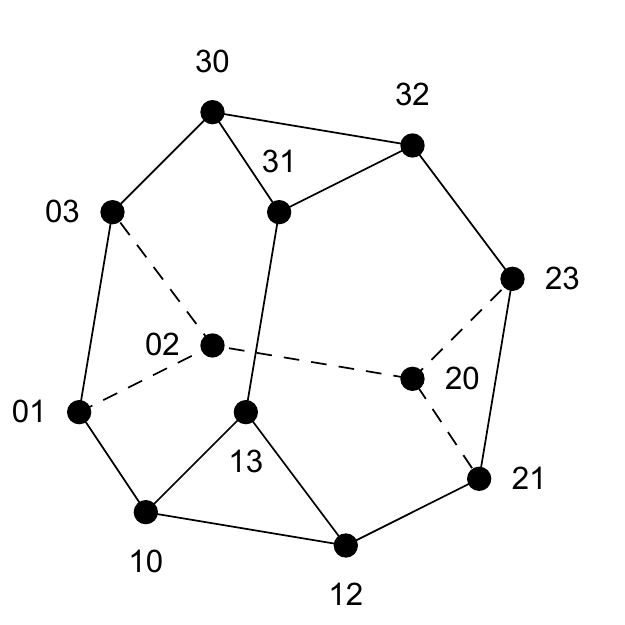}
                \caption{A truncated tetrahedron}
                \label{fig:truncation-2}
        \end{subfigure}
        \caption{An example of truncation}\label{fig:truncation}
\end{figure}

\begin{prop}
Starting with a set of $r+1$ points on the $r$-dimensional sphere and applying $k$ truncations produces a set of points such that the 1-skeleton of their convex hull is isomorphic to the graph $G_{r,k}$.
\end{prop}
\begin{proof}
We induct on $k$. When $k=1$, the graph we get is $K_{r+1}$, which is isomorphic to $G_{r,1}$.

From the geometric side, we add an auxiliary statement to the induction hypothesis: given points $p, q_1, q_2$ such that, in the associated graph, $p$ is adjacent to both $q_1$ and $q_2$, there is a 2-dimensional face of the convex hull containing all three points. This is easily verified for $k=1$.

Assuming that the induction hypotheses are true for $k-1$, fix an isomorphism of $G_{r,k-1}$ with the set of points after $k-1$ truncations, and label the points with the corresponding vertices of $G_{r,k-1}$. We claim that the graph produced after one more truncation has the following structure:
\begin{enumerate}
\item A vertex that we may label $(\vec x, \vec y)$ for every ordered pair of adjacent vertices of $G_{r,k-1}$.
\item An edge between $(\vec x, \vec y)$ and $(\vec y, \vec x)$. 

\item An edge between $(\vec x, \vec y)$ and $(\vec x, \vec z)$ whenever both are vertices of the new graph.
\end{enumerate}
The first claim is immediate from the definition of truncation: we obtain two vertices from the edge between $\vec x$ and $\vec y$. We choose to give the name $(\vec x, \vec y)$ to the vertex closer to $\vec x$. The edge between $\vec x$ and $\vec y$ remains an edge, and now joins the vertices $(\vec x, \vec y)$ and $(\vec y, \vec x)$, verifying the second claim.

By the auxiliary condition of the induction hypothesis, the vertices labeled $\vec x$, $\vec y$, and $\vec z$ lie on a common 2-face whenever $\vec x$ is adjacent to both $\vec y$ and $\vec z$. After truncation, $(\vec x, \vec y)$ and $(\vec x, \vec z)$ will also be on this 2-face; since they are adjacent along the boundary of that face, and extreme points of the convex hull, they are joined by an edge, verifying the third claim.

To finish the geometric part of the proof, we verify that the auxiliary condition remains true. There are two cases to check. For a vertex labeled $(\vec x, \vec y)$, if we choose the neighbors $(\vec x, \vec z)$ and $(\vec x, \vec w)$, then any two of them are joined by an edge, and therefore they must lie on a common 2-dimensional face. If we choose the neighbors $(\vec x, \vec z)$ and $(\vec y, \vec x)$, then the points continue to lie on the 2-dimensional face inherited from the face through $\vec x$, $\vec y$, and $\vec z$ of the previous convex hull.

Now it remains to construct an isomorphism between the 1-skeleton graph of the truncation, which we will call $T$, and $G_{r,k}$. We identify the vertex $(\vec x, \vec y)$ of $T$ with the vertex $(x_1, x_2, \dots, x_{k-1}, y_{k-1})$ of $G_{r,k}$. Since $x_{k-1} \ne y_{k-1}$ after any move in the Hanoi game, this $k$-tuple really is a Hanoi state. Conversely, any Hanoi state $\vec z \in \cH_{r,k}$ corresponds to a vertex of $T$: let $\vec x = (z_1, z_2, \dots, z_{k-1})$, and let $\vec y$ be the state obtained from $\vec x$ by either an adjustment of $z_{k-1}$ to $z_k$, if $z_k \ne z_{k-2}$, or else an involution, if $z_k = z_{k-2}$. Therefore the map we define is a bijection between the vertex sets.

Both $T$ and $G_{r,k}$ are $r$-regular graphs, therefore it suffices to show that each edge of $T$ corresponds so an edge in $G_{r,k}$. Consider an edge joining $(\vec x, \vec y)$ with $(\vec x, \vec z)$ in $T$. This corresponds to vertices $(x_1, x_2, \dots, x_{k-1}, y_{k-1})$ and $(x_1, x_2, \dots, x_{k-1}, z_{k-1})$ in $G_{r,k}$; these are adjacent, since we can obtain one from the other by an adjustment.

Next, consider an edge joining $(\vec x, \vec y)$ to $(\vec y, \vec x)$. If $\vec x$ and $\vec y$ are related by an adjustment in $G_{r,k-1}$, then they have the form $(x_1, \dots, x_{k-2}, x_{k-1})$ and $(x_1, \dots, x_{k-2}, y_{k-1})$. The vertices corresponding to $(\vec x, \vec y)$ and $(\vec y, \vec x)$ in $G_{r,k}$ are $(x_1, \dots, x_{k-2}, x_{k-1}, y_{k-1})$ and $(x_1, \dots, x_{k-2}, y_{k-1}, x_{k-1})$, and one can be obtained from the other by an involution. 

Finally, if $\vec x$ and $\vec y$ are related by an involution in $G_{r,k-1}$, then that involution swaps $x_{k-1}$ and $y_{k-1}$. Therefore such an involution in $G_{r,k}$ will take $(x_1, \dots, x_{k-1}, y_{k-1})$ to $(y_1, \dots, y_{k-1}, x_{k-1})$, and the vertices corresponding to $(\vec x, \vec y)$ and $(\vec y, \vec x)$ are adjacent in $G_{r,k}$.
\end{proof}

\bibliographystyle{plain}

\end{document}